\documentclass[]{amsart}
\usepackage{thmtools}
\usepackage[utf8]{inputenc}
\usepackage{fullpage}

\usepackage{boxedminipage}

\usepackage{listings}

\usepackage{paralist}
\usepackage{ifpdf}

\ifpdf
\usepackage[pdftex]{graphicx}
\else
\usepackage{graphicx}
\fi
\title{Regularity of infinitesimal CR automorphisms}
\author{Stefan Fürdös}
\author{Bernhard Lamel}

\newcommand{\crb}{\mathcal{V}}
\newcommand{\bcrb}{\overline{\mathcal{V}}}

\newcommand{\Cn}{\mathbb{C}^n}

\newcommand{\C}{\mathbb{C}}
\newcommand{\CN}{\mathbb{C}^N}

\newcommand{\R}{\mathbb{R}}

\newcommand{\N}{\mathbb{N}}

\newcommand{\todo}[1]{}

\newlength{\extendaxesby}

\DeclareMathOperator{\supp}{supp}

\DeclareMathOperator{\id}{Id}

\DeclareMathOperator{\spanc}{span}
\DeclareMathOperator{\imag}{Im}
\DeclareMathOperator{\real}{Re}

\DeclareMathOperator{\essup}{essupp}

\usepackage{amsthm}
\usepackage{thmtools}
\declaretheoremstyle[bodyfont=\normalfont]{noncursive}
\declaretheorem{theorem}
\declaretheorem{lemma}
\declaretheorem{proposition}
\declaretheorem{corollary}
\declaretheorem[style=noncursive]{definition}

\newcommand{\D}{\mathcal{D}}

\newcommand{\Dp}{\mathcal{D}^\prime}
\newcommand{\WF}{\mathrm{WF}\,}
\newcommand{\Char}{\mathrm{Char\,}}
\newcommand{\CotM}{T^*\Omega\!\setminus\!\{0\}}
\newcommand{\Vsmooth}{\mathcal{E}(\Omega,\C^\nu)}
\newcommand{\E}{\mathcal{E}}

\usepackage{hyperref}
\usepackage{faktor}
\begin{document}

\ifpdf
\DeclareGraphicsExtensions{.pdf, .jpg, .tif}
\else
\DeclareGraphicsExtensions{.eps, .jpg}
\fi
\begin{abstract}
We study the regularity of infinitesimal CR automorphisms of 
abstract CR structures which possess a certain microlocal extension
and show that there are smooth multipliers, completely determined by the CR structure, such that if $X$ is such an infinitesimal 
CR automorphism, then $\lambda X$ is smooth for all multipliers
$\lambda$. As an application, we study the regularity of
infinitesimal automorphisms of certain infinite type hypersurfaces
in $\Cn$. 
\end{abstract}
\maketitle

\section{Introduction and statement of results} 
\label{sec:introduction_and_statement_of_results}
Regularity of CR diffeomorphisms has been intensely studied in the real-analytic setting. We recall
here the paper of Baouendi, Jacobowitz, and Treves 	\cite{Baouendi:1985wo} stating that every smooth
CR diffeomorphism of an essentially finite real-analytic submanifold $M$ of $\CN$ extending to a wedge with edge $M$ is 
actually real-analytic. The smoothness assumptions on the map can be relaxed considerably, and only 
a certain finite smoothness will suffice in order to guarantee real analyticity of the map. 

In the setting where the regularity of the underlying manifold is reduced from real-analytic to smooth,
much less is known. Early results concentrated on the setting of strictly pseudoconvex hypersurfaces, following Fefferman's celebrated mapping theorem \cite{Fefferman:1974tl}, 
as in 
 the paper of  Nirenberg-Webster-Yang \cite{Nirenberg:1980it}, but have been based on methods which do not carry over to 
more degenerate situations. The few regularity 
results we know are regularity theorems for 
finitely nondegenerate smooth submanifolds of $\CN$ (see the paper of the second author \cite{Lamel:2004hh}), 
and more recently,
the work of Berhanu and Ming on the regularity of 
embeddings \cite{Ming:2013wq}; rougher regularity results are also implicit
in the construction of a complete system as in the work of Ebenfelt \cite{Ebenfelt:2001th}.

However, all of these results only apply to
 {\em integrable} smooth CR structures, that is, 
 CR structures
which can be realized as smooth submanifolds of some $\CN$; the recent
work of 
Berhanu and Ming actually does away with 
the requirement that the {\em source manifold} is integrable, 
but the target manifold still is required 
to be integrable; their work has actually inspired the 
research presented here.

In the current paper, we tackle the purely 
abstract setting. This requires 
us to part with all techniques relying on the use of CR functions, as 
our abstract CR structures will in general {\em not have any} solutions. However, as
they might still possess {\em symmetries}, the question 
of the regularity properties of these symmetries
 is actually interesting. Our approach to the problem is inspired by
 the approach of Berhanu and Xiao \cite{Ming:2013wq}, to which this paper owes a lot. 

Before we can state our main theorem, we need some definitions. For definitions and details regarding the notion of 
abstract CR manifolds and infinitesimal CR automorphisms, see \autoref{sec:preliminaries}. In what follows, 
we consider an abstract CR manifold $(M,\crb)$ with CR bundle $\crb \subset \C TM$. We write $\dim_\R M = 2 n + d$, where $\dim_\C \crb_p = n$
for $p\in M$, and set $N= n+d$.

\begin{definition}
 Let $(M,\crb)$ be an abstract CR manifold, and $X$ an infinitesimal CR diffeomorphism (with distributional coefficients, see \autoref{sec:preliminaries}) of $M$. 
 We say  that $X$ extends microlocally 
 to a wedge with edge $M$
 if there exists a set $\Gamma \subset T^0 M$ 
 such that for each $p\in M$, the fiber
 $\Gamma_p \subset T^0_p M\!\setminus \!\{0\}$ is a closed, convex cone, 
 and 
 \[\WF (\omega(X)) \subset \Gamma^0\]
 for every 
 holomorphic form $\omega \in \Gamma(M,T'M)$ .
 \end{definition} 

 Our first result is that there exists an ideal
  $\mathcal S \subset \mathcal{E} (M) $ of smooth functions (determined
  by the CR structure alone) such that every 
  infinitesimal CR automorphism $X$ of $M$  which extends microlocally
  to a wedge with edge $M$ has the property that $\lambda X$ is 
  smooth on $M$ for every $\lambda\in\mathcal{S}$. 

  The ideal 
  $\mathcal{S}$ is 
  constructed in the following manner. Starting with  the space
   $E_0 = \Gamma (M, T^0 M)$ we define an increasing 
   sequence of submodules $E_k \subset \Gamma (M, T'M)$ by 
   \[ E_k = \spanc_{\mathcal{E} (M)}
    \left\{ K \mapsto \omega ([L,K]) \colon 
L \in \Gamma (M, \crb), \omega \in E_{k-1}
     \right\} , \quad k\geq 1 , \quad E = \bigcup_k E_k.\] 
We then define 
\[ \mathcal{S} = \bigwedge\nolimits^N E \]
and have the following:

\begin{theorem}
\label{thm:multipliers} Let $(M,\crb)$ be an abstract, smooth CR structure, 
and $X$ an infinitesimal CR diffeomorphism of $M$ with 
distributional coefficients 
which extends microlocally to a wedge with edge $M$. 
Then, for any 
$\omega\in E$, the evaluation
$\omega(X)$ is smooth, and for any $\lambda\in\mathcal{S}$, the vector field $\lambda X  $ is smooth.
\end{theorem}

 In analogy 
to the integrable case, we will say  
that $M$ is finitely nondegenerate if 
 $\mathcal{S} = \mathcal{E} (M)$. Therefore, we have the following 
\begin{corollary}
\label{cor:finnondeg}  Let $(M,\crb)$ be an abstract, smooth, 
finitely nondegenerate CR structure, 
and $X$ a locally integrable infinitesimal CR diffeomorphism of $M$ with 
distributional coefficients 
which extends microlocally to a wedge with edge $M$. 
Then $X$ is smooth.
\end{corollary}

However, the condition that $M$ is actually finitely nondegenerate is
far too restrictive. 
We shall say that $(M,\crb)$ is CR-regular 
 if for every $p\in M$ there exists a $\lambda \in \mathcal{S}$ 
 with the property that near $p$, the zero set of $\lambda$ is 
 a real hypersurface in $M$, and such that $\lambda$ does not 
 vanish to infinite order at $p$.

\begin{theorem}
\label{thm:main} Let $(M,\crb)$ be an abstract CR structure, $p\in M$, and assume that $M$ is CR regular  near $p$. 
Then any locally integrable infinitesimal CR diffeomorphism of $M$ which 
extends microlocally to a wedge with edge $M$ is smooth. 
\end{theorem}

Without boundedness conditions on $X$, this theorem is actually 
in some sense optimal (even in the 
real-analytic case), 
as examples show (see \autoref{sec:example}).
The preceding theorem also implies a 
result in the embedded setting for  so-called ``weakly nondegenerate'' hypersurfaces. Weakly nondegenerate hypersurfaces are defined 
by the requirement
that there
exist coordinates $(z,w)\in\Cn\times\C$ and a $k\in\N $ such that $p=0$ in these
coordinates and that  near $p=0$,
$M$ is given by an equation of the form
\begin{align*}
\imag w &= (\real w)^{m} \varphi (z,\bar z, \real w),  \\
\intertext{where} 
\frac{\partial^{|\alpha| }\varphi }{\partial z^\alpha  } (0,0,0) &= \frac{\partial^{|\alpha|}\varphi }{\partial \bar z^\alpha} (0,0,0) = 0, \quad |\alpha| \leq k, 
\end{align*}
and 
\[ \spanc_\C \{ \varphi_{z{\bar z}^\alpha}  (0,0,0) \colon |\alpha|\leq k \} = \Cn.\]
If $k_0$ is the smallest $k$ for which the preceding condition holds, 
we say that $M$ is weakly $k_0$-nondegenerate $p$. 

\begin{corollary}
\label{cor:main2a}  Let $M\subset \CN$ be a 
 smooth hypersurface, $p\in M$, and 
assume that  $M$ is weakly $k$-nondegenerate  at $p$.
 Then any locally integrable infinitesimal CR diffeomorphism of $M$ which extends microlocally to a wedge with edge $M$ near $p$ is smooth 
 near $p$. 
\end{corollary}

The paper is structured as follows: In \autoref{sec:preliminaries}, we gather 
the necessary preliminaries concerning infinitesimal CR automorphisms
of abstract CR structures. In the following \autoref{sec:micro} we collect 
and prove the results of microlocal analysis which we will need. \autoref{sec:division} states a (rather simple) division theorem for smooth functions. The following
\autoref{sec:proof_of_} and \autoref{sec:proof_of_theorem_2} give the proofs of the main
results. An example illustrating the role of the multipliers is presented in 
\autoref{sec:example}.

{\bf{We would like to thank:}} an anonymous referee for a very careful reading and 
many helpful comments on the first version of the manuscript, and Shiferaw Berhanu and Michael Reiter for 
their detailed comments and helpful discussions. 

\section{Preliminaries} 
\label{sec:preliminaries}
In this section, we gather basic definitions and properties. More 
details and proofs of well known results which we do not prove here can 
be found in e.g. \cite{Baouendi:1999uy}.

An {\em{abstract CR manifold}} is a smooth real manifold $M$ together with 
a formally integrable smooth subbundle $\crb \subset \C TM$ which satisfies
$\crb \cap \bar \crb = \{0\} $. $\crb$ is called the {\em{CR bundle}} of 
$M$ and sections of $\crb$ are called {\em{CR vector fields}}. Throughout this paper,
we do not assume $M$ to be integrable, i.e. there might be no (or just a few)
\emph{solutions} of the structure or \emph{CR functions} (functions 
annihilated by all CR vector fields). $\dim_\C{} \crb = n$ is referred to 
as the \emph{CR dimension} of $M$ and we will write $\dim_\R{} M = 2n +d$.

A  map $H\colon M\supset U \to M$ of class $C^1$ is said to be CR (on $U$) if 
$dH \crb_p \subset \crb_{H(p)}$ for all $p \in U$. A vector field
 $X\colon M\supset U \to TM $ is an \emph{infinitesimal CR automorphism}
 if its local flow $H_\tau$, defined for $\tau\in\R$ small, has
  the property that for some $\varepsilon>0$, $H_\tau$ is a CR map if 
  $|\tau| < \varepsilon$. 


We will refer to the bundle $T'M:=\crb^\perp \subset \C T^* M$ as  the \emph{holomorphic cotangent bundle} of $M$, sections of $T'M$ are called
\emph{holomorphic forms}. 
Its real subbundle $T^0 M \subset T'M$,  consisting of all real dual vectors
that annihilate $\crb + \bar\crb$, is the 
\emph{characteristic bundle}. Sections of $T^0M$ are called \emph{characteristic forms}.

Let 
 $\mathfrak{Y}\in\Gamma (M, (T^{\prime} M)^{\ast})$. Recall that this is
 just the dual bundle to the space of holomorphic forms; in analogy
 to the notion of a holomorphic form, we will refer to such a $\mathfrak{Y}$ as 
 a {\em holomorphic vector field} (even though it is not holomorphic in the usual sense). Note that $ (T'M)^* = {\C TM}/{\crb}$.
 Every vector field $X\in \Gamma (M,TM) $ gives rise to a holomorphic vector field by restricting $X$ to $T'M$. 
 The following Lemma provides a converse. 

\begin{lemma}\label{lem:holo}
\label{lem:holreal} Let $\mathfrak{Y}\in\Gamma (M, (T^{\prime} M)^{\ast})$. Then 
there exists a unique real vector field $X\in \Gamma (M, TM)$ such that 
$\mathfrak{Y}$ is induced by $X$ if and only if 
\[ \mathfrak{Y} (\tau) = \overline{\mathfrak{Y} (\tau)} \]
for all characteristic forms $\tau\in\Gamma(M,T^0M)$.
\end{lemma}
Indeed, since $(\C TM)^{\ast}=\crb^\perp +\bcrb^\perp$ and $\C T^0 M=(\crb \oplus\bcrb)^\perp$, we can decompose any form
$\omega = \alpha + \bar \beta$ with $\alpha,\beta $ holomorphic forms in a nonunique manner. 
Thus $\mathfrak{Y}$ gives rise to a real vector field $X $ 
via 
\begin{equation*}
X(\omega) =
\frac{1}{2}\Bigl(\alpha \bigl(\mathfrak{Y}\bigr)+\overline{\beta \bigl(\mathfrak{Y}\bigr)}\Bigr)
\end{equation*}
which is well defined provided that $\mathfrak{Y}(\bar{\tau})=\overline{\mathfrak{Y}(\tau)}$ for all $\tau\in\Gamma(M,\C T^0 M)$ or equivalently, that
$\mathfrak{Y} (\tau)=\overline{\mathfrak{Y}(\tau)}$ for all $\tau\in  \Gamma(M,T^0 M)$, both of which are equivalent 
to the definition of $X$ above being independent of the decomposition $\omega=\alpha+\bar{\beta}$. 
We shall not distinguish between $X$ as a real vector field and as an element of $\Gamma(M,(T'M)^\ast)$.

Using the well known identity, see e.g.\ \cite{Helgason},
\[ \mathcal{L}_L \omega (K) = d\omega (L,K) + K \omega (L)  = L \omega (K) - \omega ( [L,K]),  \]
valid for arbitrary complex forms $\omega$ and complex vector fields $L,K$,
we see that the Lie derivative 
\[ \mathcal{L}_L \alpha (K) = d\alpha (L,K)  \]
of a holomorphic form $\alpha $ with respect to a CR vector field $L$ is again a holomorphic form.
We say that $\mathfrak{Y}\in\Gamma (M, (T^\prime M)^\ast)$ is \emph{CR} if
\[L \alpha(\mathfrak{Y})=d\alpha (L,\mathfrak{Y}) \]
for every CR vector field $L$ and every holomorphic form $\alpha$.  In particular,
if $X$ is a real vector field, then $X$ is CR if and only if
\[ \alpha ([L,X]) =0 \]
for every CR vector field $L$ and every holomorphic form $\alpha$.

\begin{proposition}\label{prop:inf-CR}
If $X$ is an infinitesimal CR automorphism of $M$, then $X\in \Gamma(M,(T^{\prime}M)^{\ast})$ is CR.
\end{proposition}
\begin{proof}
Let $H_\tau =\mathrm{Fl}_\tau^X$ denote the flow of X. 
By definition, $H_\tau$ satisfies the following differential equation:
\begin{equation*}
\frac{d H_\tau}{d \tau}(p)=X\circ H_\tau(p).
\end{equation*}
We note that $H_0=\id_M$ is trivially a CR map, but by assumption we know that  if $\tau$ is small then
\begin{equation*}
\omega\bigl((H_{\tau})_\ast L\bigr)=0
\end{equation*}
for any CR vector field $L$ and any holomorphic form $\omega$, i.e.\ $\omega(L)=0$.

We begin with the following general claim: For any triple $(Y,B,\alpha)$, where 
\begin{align*}
Y&=\sum_{j=1}^m Y_{j}\frac{\partial}{\partial x_{j}}\qquad Y_j\in\R\\ 
B&=\sum_{j=1}^m B_{j}\frac{\partial}{\partial x_{j}}\\
\alpha&=\sum_{j=1}^m \alpha^{j}dx^j
\end{align*}
are defined near $0$ and $\alpha (B)=0$, we have, if $K_{\tau}=\mathrm{Fl}^{Y}_{\tau}$,
\begin{equation*}
\frac{d}{d\tau}\bigl(K^{\ast}_{\tau}\alpha(B)\bigr)\bigr\vert_{\tau =0}=\alpha \bigl([B,Y]\bigr)
\end{equation*}
near the origin. For the convenience of the reader, we shall include the 
computation below.

Recalling the fact
\begin{equation*}
K^{\ast}_{\tau}\alpha\bigl(B\bigr)(p)=\alpha\bigl((K_{\tau})_{\ast}B\bigr)(K_{\tau}(p))=\sum_{j=1}^m\sum_{k=1}^m \bigl(\alpha^k\circ K_{\tau}\bigr)(p)B_j(p)\frac{\partial K^k}{\partial x_j}(p)
\end{equation*}
we can compute
\begin{equation*}
\begin{split}
\frac{d}{d\tau}\bigl(K^{\ast}_{\tau}\alpha(B)\bigr)(p)&=\sum_{j=1}^m\sum_{k=1}^m\frac{d}{d\tau}\biggl(\bigl(\alpha^k\circ K_{\tau}\bigr)(p)\frac{\partial K_{\tau}^k}{\partial x_j}(p)B_j(p)\biggr)\\
&=\sum_{j=1}^m\sum_{k=1}^m\sum_{\ell =1}^m \biggl(\frac{\partial \alpha^k}{\partial y_{\ell}}\circ K_{\tau}\!\biggr)(p)\bigl(Y_{\ell}\circ K_{\tau}\bigr)(p)\frac{\partial K_{\tau}^k}{\partial x_j}(p)B_j(p)\\
&\;\;\;\;\;+\sum_{j=1}^m\sum_{k=1}^m\sum_{\ell=1}^m\bigl(\alpha^{k}\circ K_{\tau}\bigr)(p) \biggl(\frac{\partial Y_{k}}{\partial y_{\ell}} \circ K_{\tau}\!\biggr)(p)\frac{\partial K^{\ell}_{\tau}}{\partial x_j}(p)B_j(p).
\end{split}
\end{equation*}
This leads immediately to 
\begin{equation*}
\begin{split}
\frac{d}{d\tau}\bigl(K^{\ast}_{\tau}\alpha(B)\bigr)\bigr\vert_{\tau =0}&=\sum_{k=1}^m\sum_{\ell =1}^m\biggl( \frac{\partial \alpha^k}{\partial x_{\ell}}Y_{\ell}B_{k}
+\alpha^{k} \frac{\partial Y_{k}}{\partial x_{\ell}} B_{\ell}\biggr)\\
&=\sum_{k=1}^m\sum_{\ell =1}^m\biggl(-\alpha^k Y_{\ell}\frac{\partial B_{k}}{\partial x_{\ell}}+\alpha^{k} \frac{\partial Y_{k}}{\partial x_{\ell}} B_{\ell}\biggr) \\
&=\alpha\bigl([B,Y]\bigr).
\end{split}
\end{equation*}

Now we set $Y=X$, $B=L$ and $\alpha=\omega$ as above. Then we have
\begin{equation*}
0 = \frac{d}{d\tau}\bigl(H^{\ast}_{\tau}\omega(L)\bigr)\bigr\vert_{\tau =0}=\omega\bigl([L,X]\bigr)
\end{equation*}
and hence $X$ is CR.
\end{proof}

We can now define what an 
infinitesimal CR diffeomorphism with distributional 
coefficients is. 

We say that a function $f:\, M\rightarrow\C$ is locally integrable if for any parametrization $\varphi:\,U\rightarrow M$ $f\circ\varphi$ is locally integrable on $M$. 
If $\mathsf{vol}(M)$ is the (complex) \emph{density bundle} of $M$ we define
\begin{equation*}
\D(M,\mathsf{vol}(M)):=\bigl\{\psi\in\Gamma (M,\mathsf{vol}(M))\colon\supp\subset\subset M\bigr\}
\end{equation*}
the space of compactly supported sections of $\mathsf{vol}(M)$ equipped with the usual topology.
Its strong dual $D^\prime(M)$ is the space of distributions on $M$, c.f.\ e.g.\ \cite{C-P} or \cite{MR0516965}.
A function $f:\,M\rightarrow \C$ is locally integrable if and only if 
\begin{equation*}
\int_M\!\lvert f\tau\rvert<\infty 
\end{equation*}
for all $\tau\in\D(M,\mathsf{vol}(M))$. Therefore any locally integrable function $f$ can be viewed as a distribution on $M$ in the usual way.

Furthermore we set
 \[\D(M,T'M\otimes\mathsf{vol}(M)) = \bigl\{ \omega\in \Gamma(M,T'M\otimes\mathsf{vol}(M)) \colon 
 \supp \omega \subset\subset M \bigr\}
  \]
 with the usual topology. 
 The strong dual $\D^\prime(M,(T^\prime M)^\ast):=(\D(M,T'M\!\otimes\!\mathsf{vol}(M)))^\prime$ is the space of 
 distributions (or generalized sections) on $M$ with values in $(T^\prime M)^\ast $.
 If $U\subset M$ is an open set where
  $\omega^1 , \dots ,\omega^N \in \Gamma(U,T'M)$ form a basis, 
  and $\omega_j = (\omega^j)^* \in \Gamma (U, (T'M)^*)$ is the dual basis  
  then an element $\mathfrak{Y}\in\D^\prime (M,(T'M)^\ast)$, when 
  restricted to $\D (M,(T'M\!\otimes\!\mathsf{vol}(M))|_U)$, is of the form
  \begin{equation}\label{LokRep}
  \mathfrak{Y}\vert_U=\sum_{j=1}^N c_j \omega_j,
  \end{equation} 
  where $c_j$ is a distribution on 
  $U$ for $j=1 , \dots , N$. (We also assumed that w.l.o.g.\ $U$ is small enough such that $\mathsf{vol}(M)\vert_U\cong U\times\C$.)
  We shall say that $\mathfrak{Y}\in\D^\prime (M,(T^\prime M)^\ast)$ is locally integrable if for any 
  representation of the form \eqref{LokRep} we have that $c_j$ are locally integrable functions on $U$. 
  
  We denote the usual duality bracket for $\mathfrak{Y}\in \D^\prime (M,(T^\prime M)^\ast)$ and 
$\omega\in\D(M,T^\prime\!\otimes\!\mathsf{vol}(M))$ by $\langle\mathfrak{Y},\omega\rangle\in\C$.
  
However, we can also consider a different bracket, i.e.\
\begin{equation*}
\{\,.\, ,\,.\,\}:\, \D^\prime (M,(T^\prime M)^\ast)\times \Gamma (M, T^\prime M)\longrightarrow \D^\prime (M),
\end{equation*}
which is defined locally as follows: 
On $U\subset M$ open as above we have the local representation $\eqref{LokRep}$ for $\mathfrak{Y}$ and we can write $\omega\vert_U=\sum_j f_j\omega^j$ with $f_j\in\E (U)$. We define
\begin{equation*}
\{\mathfrak{Y},\omega\}\vert_U:=f_jc_j\in\D^\prime (U).
\end{equation*}
We may write $\mathfrak{Y}(\omega)=\omega(\mathfrak{Y})=\{\mathfrak{Y},\omega\}$.
\begin{definition}
\label{def:infcrdiffeo} Let $\mathfrak{Y} \in \D^\prime(M,(T^\prime M)^\ast)$.
We say that $\mathfrak{Y}$ is an infinitesimal CR diffeomorphism
with distributional coefficients if 
\[ \mathfrak{Y} (\tau) = \overline{\mathfrak {Y} (\tau)} \]
for all $\tau \in 
\Gamma (M, T^0M)$
and 
if  $L \alpha (\mathfrak{Y}) = (\mathcal{L}_L \alpha) (\mathfrak{Y})$ for 
every $L\in\Gamma (M, \crb)$ and every $\alpha \in \Gamma(M, T'M)$.
\end{definition}

As already mentioned in the introduction, 
analogously to the integrable case, we consider
 the  increasing 
sequence  of $\E (M,\C)$ modules of forms
\begin{align*}
E_k=\bigl\langle \mathcal{L}_{K_1}\dots \mathcal{L}_{K_j}\theta\,:\;\; j\leq k,\: \: 
K_q\in \Gamma (M, \crb), \,  \theta \in \Gamma (M,T^0 M)
\bigr\rangle.
\end{align*}
We note that $E_0  = \Gamma (M, T^0 M)$,
 and $E_j \subset \Gamma(M, T' M) $ for all $j$, and 
 set $E = \bigcup_j E_j$.

We associate to the increasing chain $E_k$ the increasing sequence
of ideals $\mathcal{S}^k \subset \E (M,\C) $, where
\[  \mathcal{S}^k = \bigwedge\nolimits^N E_k =  
\left( \det \begin{pmatrix}
	V^1 ( \mathfrak{Y}_1 ) & \dots  & V^1 ( \mathfrak{Y}_N ) \\
	\vdots & & \vdots \\
	V^N ( \mathfrak{Y}_1 ) & \dots  & V^N ( \mathfrak{Y}_N ) \\
\end{pmatrix} \colon V^j \in E_k,\, \mathfrak{Y}_j \in \Gamma(M,(T'M)^*) \right).  \]  
Every $\mathcal{S}^k$ is an ideal; 
locally, one can find smaller sets of 
generators: Let $U\subset M$ be open, and  assume 
that $L_1,\dots,L_n$ is a local basis for  $\Gamma(U,\crb) $,  
that $\theta^1, \dots ,\theta^d$ is a local basis for $\Gamma(U,T^0M)$, 
and that $\omega^1, \dots, \omega^N$ is a local basis of 
$T'M$. We write 
$\mathcal{L}_j = \mathcal{L}_{L_j}$ for $j = 1, \dots, n$ and 
$\mathcal{L}^\alpha = \mathcal{L}_1^{\alpha_1} \dots \mathcal{L}_n^{\alpha_n}$
for any multi-index $\alpha = (\alpha_1, \dots, \alpha_n)\in \N^n$. 
We note that, since $\crb$ is formally integrable, the $\mathcal{L}^\alpha$, 
where $|\alpha| = k$, generate {\em all} $k$-th order homogeneous differential
operators in the $\mathcal{L}_j$, and  we thus  have 
\begin{align*}
E_k\big|_U =\bigl\langle \mathcal{L}^\alpha \theta^j \,:\;\; 1 \leq j\leq d,\: \: |\alpha| \leq k
\bigr\rangle.
\end{align*}

We can expand
\begin{equation}\label{e:exptheta} \mathcal{L}^\alpha \theta^j = \sum_{\ell=1}^{N} A^{\alpha,j}_\ell \omega^\ell \end{equation}
and for any choice
 $\underline\alpha = (\alpha^1, \dots, \alpha^N ) $  of multiindices
 $\alpha^1, \dots, \alpha^N \in \N^n $ and  $r = (r_1,\dots , r_N) \in \{1,\dots,d\}^N$  
 we define the functions
 \begin{equation}\label{equ:basisfunctions}
  D(\underline\alpha ,r)(q) = \det 
  \begin{pmatrix}
   	A^{\alpha^1,r_1}_1 & \dots & A^{\alpha^1,r_1}_N \\
   	\vdots & & \vdots \\
   	   	A^{\alpha^N,r_N}_1 & \dots & A^{\alpha^N,r_N}_N \\
   \end{pmatrix} .
\end{equation} 
With this notation, we have 
\[ \mathcal{S}^k \big|_U = \left( D(\underline\alpha ,r)\colon |\alpha^j|\leq k \right);
\]
 we shall denote the stalk of $\mathcal{S}^k$ at $p$ by $\mathcal{S}^k_{p}$.

\section{Microlocal Analysis for vector-valued Distributions}\label{sec:micro}
We gather in this section the necessary preliminary results about 
the wavefront set of sections of bundles satisfying a system of 
PDEs. 
1971 H{\"o}rmander \cite{Hoe71a,Hoe71b} introduced the notion of wavefront set. 
One of the first consequences of its definition is the microlocal elliptic regularity theorem, i.e.\
 for any distribution $u$ and (pseudo-)differential operator $P$ we have
\begin{equation}\label{PropSingSc}
\WF u\subseteq \WF Pu\cup \Char P.
\end{equation}
For a  CR distribution $v$ on a CR manifold $(M,\crb)$ the fact above amounts to saying that
 $\WF v\subset T^0M$. In order to prove the analogous fact for a CR section 
$\mathfrak{Y}$ of $ (T^{\prime}M)^{\ast}$, we need that the microlocal elliptic regularity theorem 
holds also for 
vector-valued distributions and $P$ being a square matrix of differential operators.

Indeed, a simple adaption of the arguments that establish  
\eqref{PropSingSc} in the scalar case also provides a proof in the multidimensional situation. 
However, despite relation \eqref{PropSingSc} for scalar operators being a classical result in microlocal 
analysis that is treated in numerous books e.g.\ \cite{Ho,Le,GS,Shubin} and the analogous statement 
for vector-valued distributions implicitly mentioned in the literature, 
see e.g.\ \cite{Dencker}, we were not able to find a definite source for the vector-valued case 
with precisely the statements proven  we need. 
Hence for the convenience of the reader who are not acquainted 
with microlocal analysis and pseudodifferential operators 
we try here to give a rather self-contained proof of \eqref{PropSingSc} 
for vector-valued distributions and matrix differential operators.
We mainly follow the exposition of \cite{Joshi99}, see also \cite{GS, Le, Shubin}.

Let $\Omega\subseteq\R^n$ always be an open set. 
A set $\Gamma\subseteq\R^n$ is a \emph{cone} if $\lambda\cdot x\in\Gamma$ for all 
$x\in\Gamma$ and $\lambda>0$. 
We say that a subset $V\subseteq \CotM =\Omega\times(\R^n\!\setminus\!\{0\})$ is \emph{conic}, 
if for all $(x,\xi)\in V$ and real numbers $\lambda >0$ we have $(x,\lambda\xi)\in V$. 
Sometimes we call also a conic set $V\subseteq \CotM$ a cone.
A conic neighbourhood of a point $\xi_0$ is an open cone $\Gamma$ containing $\xi_0$. 
Similarly we call an open conic set $V\subseteq \CotM$ a neighbourhood of the point 
$(x_0,\xi_0)$, if $(x_0,\xi_0)\in V$.

The space of smooth functions on $\Omega$ with values in $\C^\nu$ will be denoted by $\Vsmooth$. 
If $\nu=1$ we also write simply $\E(\Omega)$. 
As usual the space of test functions $\D (\Omega,\C^\nu)$ consists of all functions 
$f\in\Vsmooth$ with compact support.
The space of vector-valued distributions on $\Omega$ is denoted by 
$\Dp(\Omega,\C^\nu)\cong(\Dp(\Omega))^\nu$, 
whereas the space of distributions of compact support is written as 
$\E^{\prime}(\Omega)\subset\Dp(\Omega)$. 
We recall that the Fourier transform 
\begin{equation*}
\hat{v}(\xi)=\mathfrak{F}(v)(\xi)=\int\!e^{-ix\xi}v(x)\,dx
\end{equation*}
of $v\in\E^{\prime}(\Omega)$ is an analytic function. 
In general, the integral means the duality bracket for distributions.
It is well known that a distribution $v\in\E^{\prime}$ is smooth 
iff its Fourier transform $\hat{v}(\xi)$ is rapidly decreasing for $\lvert\xi\rvert\rightarrow \infty$.
Now this observation leads to definition of the \emph{wavefront set} of a distribution $u\in\Dp(\Omega)$.
\begin{definition}
Let $(x_0,\xi_0)\in \CotM$. 
The point $(x_0,\xi_0)$ is not in $\WF u$ iff there is $\varphi\in\D(\Omega)$ with $\varphi\equiv 1$ near $x_0$ and a conic neighbourhood $\Gamma$ of $\xi_0$ such that
\begin{equation}\label{WFest}
\sup_{\xi\in\Gamma}\lvert\xi\rvert^N\lvert\widehat{\varphi v}(\xi)\rvert<\infty \qquad \forall N.
\end{equation}
\end{definition}
We note that if $(x_0,\xi_0)\notin\WF u$ then $(x_0,-\xi_0)\notin \WF\bar{u}$, here the conjugate $\bar{u}$ of $u\in\Dp$ is defined by $(\bar{u},\varphi)=(u,\bar{\varphi})$.

The wavefront set of $u=(u_1,\dots ,u_\nu)\in\Dp (\Omega,\C^\nu)$ is then defined as $\WF u=\bigcup_{j=1}^\nu \WF u_j$. 
Obviously we have the following characterization of $\WF u$:
A point $(x_0,\xi_0)\in \CotM$ is not in $\WF u$ iff there is a test function $\varphi\in\D(\Omega)$ with $\varphi\equiv 1$ near $x_0$ such that \eqref{WFest} for each $\widehat{\varphi u_j}$ on a common conic neighbourhood $\Gamma$ of $\xi_0$.
It is easy to see that $\WF u$ is a conic set.

Before we can finally begin with the proof of \eqref{PropSingSc}, we have to introduce matrix-valued pseudodifferential operators. 
We may assume that the theory of scalar-valued pseudodifferential operators on open sets is known to the reader, a detailed introduction can be found in \cite{Ho, Le, GS}.

A pseudodifferential operator $P$ of order $m$ is an operator $P: \D (\Omega,\C^\nu)\rightarrow \Vsmooth$ of the form 
\begin{equation}
P=\begin{pmatrix}
P^{11} & \dots  &P^{1\mu}\\
\vdots & \ddots & \vdots \\
P^{\nu 1} & \dots & P^{\nu\mu} 
\end{pmatrix}
\end{equation}
where $P^{jk}\in\Psi_{0,1}^m(\Omega)$ are scalar pseudodifferential operators of order $m$.
The symbol of $P$ is the matrix
\begin{equation}
p(x,\xi)=\begin{pmatrix}
p^{11}(x,\xi) & \dots & p^{1\mu}(x,\xi)\\
\vdots & \ddots & \vdots \\
p^{\nu 1}(x,\xi) & \dots & p^{\nu\mu}(x,\xi) 
\end{pmatrix}
\end{equation}
whose entries $p^{jk}\in S^m(\Omega\times\R^n)$ are the symbols of the operators $P^{jk}$. 
The class of pseudodifferential operators of order $m$ between vector-valued functions is denoted by $\Psi^m(\Omega ,\C^\nu)$. 
As in the case of scalar operators the elements of the set $\Psi^{-\infty}(\Omega)=\bigcap\Psi^m(\Omega)$ are called smoothing operators; 
if $Q\in\Psi^{-\infty}(\Omega)$ then $Q(\E^{\prime}(\Omega,\C^\nu))\subseteq \Vsmooth$.
As with scalar operators we can associate to each pseudodifferential operator $P$ a properly supported pseudodifferential operator $\tilde{P}$, i.e.
\begin{align*}
 \tilde{P}: \E^{\prime}(\Omega,\C^\nu)&\rightarrow \E^{\prime}(\Omega,\C^\nu) \\
\intertext{and}
\tilde{P}: \Dp(\Omega,\C^\nu)&\rightarrow \Dp(\Omega ,\C^\nu)
\end{align*}
respectively, with  $P-\tilde{P}\in\Psi^{-\infty}$ since we can repeat the procedure in the scalar case (see e.g.\cite{Le}) in each entry separately to obtain the desired operator.
Similarly we can construct to each sequence $a_j\in S^{m-j}$ a symbol $a\in S^m$ such that $a-\sum_{j<N}a_j\in S^{m-N}$ by also repeating the proof from the scalar case, cf.\ \cite{Le}. 
We will use the notation $a\sim\sum a_j$.

The composition of two properly supported pseudodifferential operators $A ,B$ of order $m_1$ and $m_2$ respectively is the operator $C$ given by the matrix with entries
\begin{equation*}
C^{j\ell}=\sum_{k=1}^\nu A^{jk}B^{k\ell}.
\end{equation*}
For the symbol $c$ of $C$ we write $a\sharp b$. 
We have that the symbol $a^{jk}\sharp b^{k\ell}$ of $A^{jk}\circ B^{k\ell}$ must satisfy the following expansion (c.f.\ \cite{Le}) 
\begin{equation*}
a^{jk}\sharp b^{k\ell}\sim\sum_{\alpha}\frac{\partial^{\alpha}_\xi a^{jk}(x,\xi)D_x^{\alpha}b^{k\ell}(x,\xi)}{\alpha !},
\end{equation*} 
hence
\begin{equation*}
c^{j\ell}\sim\sum_{k=1}^\nu\sum_{\alpha}\frac{\partial^{\alpha}_\xi a^{jk}(x,\xi)D_x^{\alpha}b^{k\ell}(x,\xi)}{\alpha !}.
\end{equation*}
We see that the analogous formula has to be valid in the matrix case
\begin{equation}
a \sharp b\sim\sum_{\alpha}\frac{\partial^{\alpha}_\xi a(x,\xi)D_x^{\alpha}b(x,\xi)}{\alpha !}.
\end{equation}

As in the case of scalar operators (see e.g.\ \cite{GS}) we say that a properly supported operator $P\in\Psi^m_{ps}(\Omega, \C^\nu)$
is a classical pseudodifferential operator if there are smooth functions $p_{m-j}$ on $\CotM$ that are homogenous of degree $m-j$ in the second variable
and $\psi\in\D({\R^n})$ with $\psi\equiv 1$ near the origin such that the symbol $p$ of $P$ satisfies the following asymptotic expansion
\begin{equation*}
\psi(x,\xi)\sim\sum_{j} (1-\psi(\xi))p_{m-j}(x,\xi).
\end{equation*}
In slight abuse of notation we write also
\begin{equation}
p\sim\sum p_{m-j}
\end{equation}
and we will sometimes refer to the (formal) series $\sum p_{m-j}$ as the symbol of $P$. 
The term of highest order $p_m$ in the series is called the principal symbol of $P$.
The class of classical operators of order $m$ will be denoted by $\Psi_{cl}^m(\Omega,\C^\nu)$ and the 
term of highest order in the asymptotic expansion is called the principal symbol of the operators.
If $A,B$ are two classical pseudodifferential operators of order $m_1$ and $m_2$, resp.\ , 
then we have that $C=A\circ B$ is a classical operator of order $m_1+m_2$ and,
 if $c\sim\sum c_{m-\ell}$ and $m=m_1+m_2$,
\begin{equation}
c_{m-\ell}=\sum_{j+k+\lvert\alpha\rvert=\ell}\frac{1}{\alpha!}\partial_{\xi}^{\alpha}a_{m_1-j}(x,\xi)D_x^{\alpha}b_{m_2-k}(x,\xi).
\end{equation}
We see that for the principal symbols, i.e.\ $\ell=0$, the above equation is just $c_m=a_{m_1}b_{m_2}$.

We close this very short introduction with two definitions, that are completely analogous 
to the definitions for scalar operators, see \cite{Le}.
\begin{definition}
The essential support $\mathrm{essupp}\, A\subseteq \CotM$ of $A\in\Psi^m_{ps}(\Omega,\C^\nu)$
 is defined by saying that a point $(x_0,\xi_0)$ is not in $\mathrm{essupp}\, A$ 
 if there is a conic neighbourhood of $(x_0,\xi_0)$ such that
\begin{equation}
\sup_{(x,\xi)\in\Gamma}\Bigl\lvert\bigl(\partial_{\xi}^\alpha\partial_x^\beta a^{jk}\bigr)\Bigr\rvert(x,\xi)\lvert\xi\rvert^N<\infty
\end{equation}
for all $\alpha,\beta\in\N_0^n$, $N\in\N$ and $j,k=1,\dots, \nu$.
\end{definition}
If $A,B\in\Psi^{\infty}_{ps}(\Omega,\C^\nu)$ then 
$\mathrm{essupp}\, AB\subseteq \mathrm{essupp}\,A\cap\mathrm{essupp}\,B$.
\begin{definition}
An operator $A\in\Psi_{cl}^m(\Omega ,\C^\nu)$ is elliptic or non-characteristic 
at $(x_0,\xi_0)\in \CotM$ if the principal symbol $a_m$ of $A$ is invertible at $(x_0,\xi_0)$. 
We set
\begin{equation}
\Char A:=\bigl\{(x,\xi)\in \CotM\mid a_m(x,\xi)\text{ is not invertible}\bigr\}
\end{equation}
\end{definition}

Now we are able to start with the proof of \eqref{PropSingSc} for vector-valued distributions.
\begin{theorem}\label{micro-parametrix}
Let $P\in\Psi_{cl}^m(\Omega,\C^\nu)$ be elliptic at $(x_0,\xi_0)\in \CotM$.
Then there are operators $Q\in\Psi^{-m}_{cl}(\Omega,\C^\nu)$
 and $R,S\in\Psi^0_{cl}(\Omega ,\C^\nu)$ such that
\begin{align*}
QP&=\id +R & (x_0,\xi_0)&\notin\essup R\\
PQ&=\mathrm{Id} +S & (x_0,\xi_0)&\notin\mathrm{essupp}\, S
\end{align*}
\end{theorem}
\begin{proof}
We set $\tilde{q}_{-m}=(p_m)^{-1}$ in some conic neighbourhood $\Gamma$ of $(x_0,\xi_0)$ 
where $\det p_m \neq 0$.
Recursively we define on $\Gamma$
\begin{equation*}
\tilde{q}_{-m-N}(x,\xi)=-(p_m(x,\xi))^{-1}\negthickspace\sum_{\substack{j+k+\lvert\alpha\rvert=N \\ 
j\leq N-1}} \frac{1}{\alpha !}\partial_{\xi}^\alpha \tilde{q}_{m-j}(x,\xi)D_x^{\alpha}p_{m-k}(x,\xi)
\end{equation*}
Using a suitable cut-off function $\psi\in\E(\CotM)$, i.e.\ $\mathrm{supp}\,\psi\subseteq\Gamma$, 
$\psi\equiv 1$ near $(x_0,\xi_0)$ and $\psi$ is homogeneous of degree $0$ in the second variable, 
we can extend the functions $\tilde{q}_{-m-k}$ to the whole space $\CotM$ 
by putting $q_{-m-k}=\psi \tilde{q}_{-m-k}$. 
Let $Q$ be the classical pseudodifferential operator associated to the symbol $\sum q_{-m-k}$. 
Then clearly $QP\in\Psi_{cl}^0(\Omega,\C^\nu)$ and therefore 
$R:=\mathrm{Id}-QP\in\Psi^0_{cl}(\Omega,\C^\nu)$.
If $\sum r_{-j}$ is the symbol of $R$ then it follows that $r_{-j}\equiv 0$ 
in some conic neighbourhood of $(x_0,\xi_0)$ by construction.
Hence $(x_0,\xi_0)\notin\mathrm{essupp}\, R$.

Analogously we can construct $Q_1\in\Psi^{-m}_{cl}(\Omega,\C^\nu)$ 
such that $PQ_1=\mathrm{Id}+S_1$ with $S_1\in\Psi^{0}_{cl}(\Omega ,\C^nu)$ and 
$(x_0,\xi_0)\notin\mathrm{essupp}\,S_1$. 
Following an argument in \cite{Le} we conclude that
\begin{equation*}
Q=Q(PQ_1-S_1)=(\mathrm{Id}+R)Q_1-QS_1=Q_1+RQ_1-QS_1
\end{equation*}
and
\begin{equation*}
PQ=PQ_1+PRQ_1-PQS_1=\mathrm{Id} +S_1+ PRQ_1-PQS_1=\mathrm{Id}+ S
\end{equation*}
where $S=S_1+PRQ_1-PQS_1\in\Psi^0_{cl}(\Omega,\C^\nu)$.
 Clearly $(x_0,\xi_0)\notin\mathrm{essupp}\, S$.
\end{proof}

\begin{proposition}\label{WFChar}
Let $u\in\Dp(\Omega,\C^\nu)$. Then we have
\begin{equation*}
\WF u=\negthickspace\negmedspace\bigcap_{\substack{P\in\Psi^{\infty}(\Omega ,\C^\nu)\\
 Pu\in\Vsmooth}}\negthickspace\negthickspace\negthickspace\negthickspace \Char P.
\end{equation*}
\end{proposition}
\begin{proof}
If $(x_0,\xi_0)\notin \WF u$ then there is a test function $\varphi\in\D (\Omega)$ 
such that $\widehat{\varphi u_k}$ is rapidly decreasing for each $k=1,\dots,\nu$ 
on an open cone containing $\xi_0$. 
The multiplication with $\varphi$ is a differential operator, that we will denote 
by $\Phi$ in the scalar case.
If $v\in\Dp(\Omega, \C^\nu)$ the operator $\Phi_\nu$ acting on $\Dp(\Omega, \C^\nu)$ 
given by $(\Phi_\nu v)_j=\varphi v_j$ (i.e.\ $\Phi_\nu=\Phi\cdot\mathrm{Id}$) is 
also a differential operator of order $0$.

Let $\psi\in\E(S^{n-1})$ such that $\psi\equiv 1$ near $\xi_0/\lvert\xi_0\rvert$ and 
$\mathrm{supp}\,\psi\subset\subset\Gamma\cap S^{n-1}$ and set
\begin{equation*}
a(x,\xi)=\varphi (x)\psi \biggl(\frac{\xi_0}{\lvert\xi_0\rvert}\biggr)
\end{equation*}
and $A\in\Psi_{cl}^0(\Omega)$ with
\begin{equation*}
Aw (x)=\frac{1}{(2\pi)^n}\int\!e^{ix\xi}a(x,\xi)\hat{w}(\xi)\,d\xi
\end{equation*}
for $w\in\D(\Omega)$.

We define $B:=A\cdot \mathrm{Id}\in\Psi_{cl}^0(\Omega,\C^\nu)$ and $Q:=B\circ\Phi_\nu$. 
Now $Q$ acts on $u$ by
\begin{equation*}
(Qu)_k=\frac{1}{(2\pi)^n}\int\!e^{ix\xi}a(x,\xi)\widehat{\varphi u}(\xi)\,d\xi=\frac{1}{(2\pi)^n}\varphi(x)\int\!e^{ix\xi}\psi\biggl(\frac{\xi}{\lvert\xi\rvert}\biggr)\widehat{\varphi u}(\xi)\,d\xi
\end{equation*}
where the integral is for the moment only seen as $\mathfrak{F}^{-1}(a(x,\,.\,\widehat{\varphi u}(\, .\,))$.
 Since by assumption
\begin{equation*}
G_k(\xi)=\psi\biggl(\frac{\xi}{\lvert\xi\rvert}\biggr)\widehat{\varphi u}(\xi)\in\mathcal{S}(\R^n)
\end{equation*}
we have $(Qu)_k\in\D(\Omega)\subseteq\E(\Omega)$. By construction $(x_0,\xi_0)\in\Char Q$.

Now let $P\in\Psi^m_{cl}(\Omega,\C^\nu)$ with $(x_0,\xi_0)\notin\Char P$ and 
$Pu\in\mathcal{E}(\Omega,\C^\nu)$. 
According to Theorem \ref{micro-parametrix} there is an operator 
$R\in\Psi^0_{cl}(\Omega,\C^\nu)$ with $(x_0,\xi_0)\notin \mathrm{essupp} R$ and
\begin{equation*}
u+Ru\in\Vsmooth.
\end{equation*}
 If we choose $\varphi$ and $\psi$ similarly to above and set 
 $\theta(x,\xi)=(\psi(\xi) \sharp\varphi(x))I$ we can assume that 
 $\mathrm{supp}\,\theta\cap\mathrm{essupp}\,R=\emptyset$. 
Therefore $\Theta R\in\Psi^{-\infty}$ and $\Theta u\in\Vsmooth$.  Actually we have  as above 
\begin{equation*}
\bigl(\Theta u\bigr)_k=\frac{1}{(2\pi)^n}\int\! e^{ixy}\psi(\xi)\widehat{\varphi u}(\xi)\,d\xi
\end{equation*}
By Lemma A.1.2 in \cite{Shubin} we have that $(\Theta u)_k\in\mathcal{S} (\R^n)$ and hence 
$\psi\widehat{\varphi u}\in\mathcal{S}(\R^n)$.
It follows promptly that $\widehat{\varphi u}$ has to decrease rapidly
 in a conic neighbourhood of $\xi_0$.
\end{proof}

\begin{proposition}\label{PropWF}
Let $u\in\Dp (\Omega,\C^\nu)$ and $P\in\Psi^m_{cl}(\Omega ,\C^\nu)$. Then
\begin{equation*}
\WF Pu\subseteq \WF u\cap \mathrm{essupp}\, P
\end{equation*}
\end{proposition}
\begin{proof}
Let $(x_0,\xi_0)\notin\WF u$. By Proposition \ref{WFChar} there is a 
classical pseudodifferential operator $Q_1$ with $(x_0,\xi_0)\notin\Char Q_1$ 
such that $Q_1u\in\Vsmooth$. 
Theorem \ref{micro-parametrix} in turn provides an operator $Q_2$ such that
 $Q=Q_2Q_1=\mathrm{Id}+R$ where $R\in\Psi^0_{cl}(\Omega,\C^\nu)$ and 
 $(x_0,\xi_0)\notin\mathrm{essupp}\,R$.
Apparently $Qu\in\E$, hence $PQ\in\E$. 
On the other hand 
\begin{gather*}
QPu=PQu+[Q,P]u\\
\intertext{and}
[Q,P]=(\mathrm{Id}+R)P-P(\mathrm{Id}+R)=RP-PR=[R,P]
\end{gather*}
Thus $(x_0,\xi_0)\notin\mathrm{essupp}\,[Q,P]$.
As in the proof of Proposition \ref{WFChar} we construct a classical pseudodifferential operator $S$ 
that is elliptic at the point $(x_0,\xi_0)$
and satisfies $\mathrm{essupp}\,S\cap\mathrm{essupp}\,R=\emptyset$.
It follows that $\mathrm{essupp}\,S\cap\mathrm{essupp}\,[Q,P]=\emptyset$ and therefore 
$S[Q,P]u\in\E$ and $SQPu\in\mathcal{E}$.
The operator $SQ$ is non-characteristic at $(x_0,\xi_0)$.
Hence $(x_0,\xi_0)\notin\WF Pu$.

Now let $(x_0,\xi_0)\notin\mathrm{essupp}\,P$.
Again we construct an operator $S$ satisfying $(x_0,\xi_0)\notin\Char S$ and 
$\mathrm{essupp}\,S\cap\mathrm{essupp}\,P=\emptyset$.
Thus $SP\in\Psi^{-\infty}$ and $SPu\in\E$. It follows $(x_0,\xi_0)\notin\WF (Pu)$.
\end{proof}
\begin{theorem}\label{thm:vec-prop}
If $P\in\Psi^{m}_{cl}(\Omega ,\C^\nu)$ then we have for all $u\in\Dp(\Omega,\C^\nu)$ that
\begin{equation}
\WF u\subseteq\WF (Pu)\cup \Char P.
\end{equation}
\end{theorem}
\begin{proof}
If $(x_0,\xi_0)\notin \WF(Pu)\cup\Char P$ then $P$ is elliptic at the point $(x_0,\xi_0)$. 
By Theorem \ref{micro-parametrix} there are an operator $Q\in\Psi^{-m}_{cl}$ elliptic at $(x_0,\xi_0)$ 
and an operator $R\in\Psi^0_{cl}$ with $(x_0,\xi_0)\notin\mathrm{essupp}\,R$ such that 
$QP=\mathrm{Id}+R$. 
We have $(x_0,\xi_0)\notin\WF(QPu)$ by Proposition \ref{PropWF}. 
On the other hand $QPu=u+Ru$ and $(x_0,\xi_0)\notin\WF (Ru)$ again by Proposition \ref{PropWF}.
Hence $(x_0,\xi_0)\notin \WF u$.
\end{proof}
\section{A Division Theorem}\label{sec:division}
The aim of this section is to study the following question: 
Suppose that $\lambda$ is a smooth function and $u$, say, a locally 
integrable function such that $f=\lambda\cdot u$ is smooth. 
Can we conclude that $u$ itself has to be smooth?
Obviously this is a local problem and the only points of interest are the zeros of $\lambda$, 
since $u=f/\lambda$ must be smooth whenever $\lambda\neq 0$. 

On the other hand the example
 $\lambda (x) =e^{-1/x^2}$ and $u=\lvert x\rvert$  shows that 
 $u$ may have singularities at the points where $\lambda$ is flat, 
and furthermore the example $\lambda (x,y) = x^2 + y^2$ shows that 
the structure of the zero set of $\lambda$ is of importance. 

We are going to only give a simple sufficient condition on $\lambda$
adapted to the applications which we have in mind.  
It remains to study the situtation near zeros of finite order of $\lambda$. 
We begin with the one-dimensional case.

\begin{lemma}\label{lem:division}
Let $\lambda$ be a smooth function near $0\in\R$ such that there is some $k\in\N$ 
with $\lambda^{j}(0)=0$ for $0\leq j<k$ and $\lambda^{(k)}(0)\neq 0$. 
Furthermore let $u$ be locally integrable near $0$ such that $f:=\lambda u$ is smooth near $0$.
Then $u$ is smooth near $0$, too.
\end{lemma}
\begin{proof}
First, we note that the zero of $\lambda $ at $0$ is isolated. 
By the Fundamental Theorem of Calculus we obtain easily the existence of a smooth function 
$\tilde{\lambda}$ with $\tilde{\lambda}(0)\neq0$ such that 
\begin{equation*}
\lambda(x)=x^k\tilde{\lambda}(x)
\end{equation*}
near the origin.

In order to proceed we need a similar decomposition for $f$.
But, since we do not know the values of the derivatives of $f$  at the origin a-priori,
the Fundamental Theorem of Calculus only says that there is a smooth function $f_1$ 
such that $f=xf_1$. 
If $k>1$ then in a punctured neighbourhood  of $0$ we have
\begin{equation*}
u(x)=x^{1-k}\frac{f_1(x)}{\tilde{\lambda}(x)}
\end{equation*}
and if $f_1(0)\neq 0$ then $u(x)\sim x^{1-k}$ for $x\rightarrow 0$.
This is a contradiction to $u$ being locally integrable.
Therefore $f_1(0)=0$ and there is a smooth function $f_2$ near the origin such that $f(x)=x^2f_2(x)$.

Iterating this argument if necessary we obtain that there is a smooth function $f_k$ near $0$ 
such that $f(x)=x^kf_k(x)$.
Hence we obtain in some punctured neighbourhood of $0$ the following representation of $u$
\begin{equation*}
u(x)=\frac{f_k(x)}{\tilde{\lambda}(x)},
\end{equation*}
where the right-hand side of this equation can be extended smoothly to the origin.
\end{proof}

One cannot expect that the analogous result to \autoref{lem:division} holds in several variables 
(c.f.\ \cite{Boman}). 
However, one can adapt the proof of \autoref{lem:division} to show a partial result 
for smooth functions whose zero set satisfies additionally certain geometric conditions.

\begin{proposition}\label{prop:division}
Let $p_0\in\R^n$ and $\lambda$ a smooth function defined near $p_0$.
Suppose that $\lambda^{-1}(0)$ is a real hypersurface in $\R^n$ near $p_0 \in \lambda^{-1}(0)$ 
and that there are $v\in\R^n$ and $k\in\N$ such that $\partial_{v}^j\lambda (p) =0$, for $j<k$ and
$p\in\lambda^{-1} (0)$ close by $p_0$, and
$\partial^k_v\lambda (p_0)\neq 0$.

If $u$ is a locally integrable function near $p_0$ with the property that $f=\lambda\cdot u$ is smooth, 
then $u$ has to be smooth near $p_0$.
\end{proposition}
\begin{proof}
We can choose coordinates $(x_1,\dots, x_n)=(x^{\prime},x_n)$ in a neighbourhood $V$ of $p_0$ 
such that $p_0=0$ in these coordinates, $\lambda^{-1}(0)\cap V=\{(x^\prime ,x_n)\in V\mid x_n=0\}$
 and for $(x^\prime,0)$ we have
\begin{align*}
\frac{\partial^j  \lambda}{\partial x_n^j}(x^{\prime},0) &=0, & j&<k,\\
\intertext{and}
\frac{\partial^k \lambda}{\partial x_n^k}(x^\prime ,0) &\neq 0.
\end{align*}
As in the proof of \autoref{lem:division} we conclude, if we shrink $V$, that there is a smooth function 
$\tilde{\lambda}$ on $V$ with $\tilde{\lambda}(x)\neq 0$ for $x\in V$ 
such that $\lambda (x^\prime ,x_n)=x_n^k\tilde{\lambda}(x^\prime,x_n)$. 
There is also a smooth function $f_1$ on $V$ such that $f(x^\prime,x_n)=x_nf_1(x^\prime,x_n)$.
We want to show as in the one-dimensional case that $f_1(x^\prime, 0)=0$ 
for $(x^\prime,0)\in V$ if $k>1$:
Suppose that there exists some $y\in\R^{n-1}$ with $(y,0)\in V$ and $f_1(y,0)\neq 0$.
 Then there is a neighbourhood $W$ of $(y,0)$ such that $f_1(x)\neq 0$ 
 and also $\tilde{\lambda}(x)\neq 0$ for $x\in W$.
W.l.o.g.\ the open set $W$ is of the form $W=W^{\prime}\times I\subset \R^{n-1}\times \R$ and
\begin{equation*}
F(x_n):=\int\limits_{W^\prime}\biggl\lvert\frac{f_1}{\tilde{\lambda}}(x^\prime, x_n)\biggr
\rvert\,dx^\prime>0
\end{equation*}
for $x_n\in I$. We conclude that
\begin{equation*}
\int\limits_W\lvert u(x)\rvert\,dx=\int\limits_I \lvert x_n\rvert ^{1-k}F(x_n)\,dx_n =\infty
\end{equation*}
and hence $u$ is not locally integrable near $(y,0)$ which contradicts our assumption.

Therefore  we obtain by iteration a smooth function $\tilde{f}$ defined near the origin in $\R^n$ such that  
$f(x^\prime,x_n)=x_n^k\tilde{f}(x^\prime,x_n)$.
Hence $u=\tilde{f}/\tilde{\lambda}$ is also smooth near the origin.
\end{proof}

\section{Proof of Theorem~1} 
\label{sec:proof_of_}

Let $X$ be an infinitesimal CR diffeomorphism as in the 
statement of the theorem. The assertion of the theorem can be checked
locally, so we restrict ourselves to an open set $U\subset M$ on 
which we are given a basis $L_1, \dots, L_n$ of CR vector fields, a 
basis $\omega^1 , \dots , \omega^N$ of holomorphic forms, and a 
generating set $\theta^1 , \dots , \theta^d $ of characteristic forms.
We also assume that $X$ extends microlocally to a wedge with edge $M$.

Since $\mathcal{L}_{L_k}$ maps holomorphic forms to 
holomorphic forms, we can write 
\[ d\omega^j (L_k, \cdot) = \sum_{\ell=1}^N B^j_{k,\ell} \omega^\ell (\,.\, ) \]
for functions $B^j_{k,\ell}$ which are smooth on $U$. 
By \autoref{lem:holo} we can regard $X$  as a holomorphic vector field and write 
\[ X = \sum_{j=1}^N X_j (\omega^j)^*. \]
The assumption that $X$ is an infinitesimal CR diffeomorphism implies that
\begin{gather} 
L_k X_j =  L_k (\omega^j (X)) = d\omega^j(L_k, X) = 
\sum_{\ell=1}^N B^j_{k,\ell} X_\ell  \label{e:CR1}\\
\intertext{and}
\theta (X)=\overline{\theta(X)}\label{e:inf1}
\end{gather}
for any characteristic form $\theta$.

The proof of \autoref{thm:multipliers} follows now from the next statement.
\begin{proposition}\label{prop:main}
Let $(M,\crb)$ be an abstract, smooth CR structure and let $X$ be a distributional 
CR holomorphic vector field; that is a distributional section 
of $(T^\prime M)^*$ that satisfies \eqref{e:CR1}. 
Assume further that $X$ fulfills \eqref{e:inf1} for any characteristic form and extends microlocally 
to a wedge with edge $M$.
Then for any $\lambda\in\Gamma(M,\mathcal{S}^k)$ the section $\lambda X$ is smooth.
\end{proposition}
\begin{proof}
Similarly to above we can write locally on $U\subset M$
\begin{equation*}
X=\sum_{j=1}^NX_j (\omega^j)^*
\end{equation*}
where now $X_j\in\D^{\prime}(U)$.

If we consider the vector
\[ \tilde X = ( \underbrace{X_1, \dots, X_1}_{{n} \text{ times}} , \dots, \underbrace{X_N, \dots, X_N}_{{n} \text{ times}} ),\]
then \eqref{e:CR1} implies that $\tilde X$ satisfies the equation
$P \tilde X = K \tilde X$, 
where
\[ 
P = \begin{pmatrix}
 \begin{matrix}
 	L_1 & 0 & \dots & 0 \\
 	0 & L_2 &  \dots & 0 \\
 	\vdots & \vdots & \ddots &  \\
 	0 & 0 & \dots & L_n
 \end{matrix} & \dots & 0 \\
 \vdots & \ddots &  \\
 0 & & \begin{matrix}
 	L_1 & 0 & \dots & 0 \\
 	0 & L_2 &  \dots & 0 \\
 	\vdots & \vdots & \ddots & 0 \\
 	0 & 0 & \dots & L_n
 \end{matrix}
\end{pmatrix}
\]
and 
\[ K = \begin{pmatrix}
	B^1_{1,1} & 0 & \dots & B^1_{1,2} & 0 & \dots & \dots & B^1_{1,N} & \dots & 0 \\
	B^1_{2,1} & 0 & \dots & B^1_{2,2} & 0 & \dots & \dots & B^1_{2,N} & \dots & 0 \\
	\vdots & \vdots &  & \vdots & \vdots &  &  & \vdots &  & \vdots \\
	B^1_{n,1} & 0 & \dots & B^1_{n,2} & 0 & \dots & \dots & B^1_{n,N} & \dots & 0 \\
\vdots & \vdots &  & \vdots & \vdots &  &  & \vdots &  & \vdots \\
B^N_{1,1} & 0 & \dots & B^N_{1,2} & 0 & \dots & \dots & B^N_{1,N} & \dots & 0 \\
	B^N_{2,1} & 0 & \dots & B^N_{2,2} & 0 & \dots & \dots & B^N_{2,N} & \dots & 0 \\
	\vdots & \vdots &  & \vdots & \vdots &  &  & \vdots &  & \vdots \\
	B^N_{n,1} & 0 & \dots & B^N_{n,2} & 0 & \dots & \dots & B^N_{n,N} & \dots & 0 \\
\end{pmatrix}. \]
We can thus apply \autoref{thm:vec-prop} in order
to see that the components $X_j$ of $X$ have their wavefront sets
contained in the characteristic directions of the $L_k$, which 
means, they are restricted to $T^0 M$. By assumption, we know 
that there exists a closed convex cone $\Gamma^0  \subset T^0 M \setminus \{0\}$ such that 
$ \WF (X_j) \subset\Gamma^0 $ for every $j=1,\dots,N$. If we denote
by $W^+ = (\Gamma^0)^c \subset T^0 M $  then we have
$\WF (X_j) \cap W^+  =\emptyset $ for all $j = 1, \dots ,N$. 
For simplicity, we shall say that $X_j$ {\em extends above}; 
similarly, with $(-\Gamma^0)^c = W^-$, we have that $\WF (\bar X_j) \cap W^- = \emptyset$, and say that $\bar X_j$ 
\emph{extends below}. 
Obviously the same is true for any derivative of the $X_j$
or $\bar X_j$, respectively. 

By \eqref{e:inf1} we know that
$\theta (X) = \overline{ \theta(X)}$ for every characteristic form $\theta$. Recall from \eqref{e:exptheta} that we write 
\[ \mathcal{L}^\alpha \theta^j = \sum_{\ell=1}^{N} A^{\alpha,j}_\ell \omega^\ell; \]
hence, in coordinates, the equation $\theta(X) = \overline{\theta(X)}$ becomes  
\begin{equation}
	\label{e:refl1} 
 \sum_{\ell=1}^N A^{0,j}_\ell X_\ell = 
\sum_{\ell=1}^N \overline{A^{0,j}_\ell} \bar{X}_\ell, 
\end{equation}
and the left hand side of that equation extends below, 
while the right hand side extends above. Choose any $N$-tuple 
$\underline{\alpha} = (\alpha^1 , \dots , \alpha^N )\in \N_0^{Nn}$, 
with $|\alpha^j|\leq k $ for $j=1,\dots, N$, and
$r= (r_1, \dots ,r_N)\in \{1,\dots,d\}^N$. Applying $\mathcal{L}^\alpha$ to $\theta(X) = \overline{\theta(X)}$ (i.e. to \eqref{e:refl1}) gives
\[ \sum_{\ell = 1}^N A^{\alpha,j}_\ell X_\ell = \sum_{\substack{|\beta|\leq |\alpha| \\ \ell=1,\dots, N}}
C^{\alpha, \ell}_\beta L^\beta \bar{X}_\ell,
 \]
where $C^\ell_\beta$ are smooth functions on $U$. We thus have, for 
the above choice of $\underline{\alpha}$ and $r$, the following 
system of equations:
\[ 
\begin{pmatrix}
	A^{\alpha^1,r_1}_1 & \dots & A^{\alpha^1,r_1}_N \\
	\vdots &\ddots & \vdots\\
	A^{\alpha^N,r_N}_1 & \dots & A^{\alpha^N,r_N}_N 
\end{pmatrix}
\begin{pmatrix}
	X_1 \\ \vdots \\ X_N
\end{pmatrix} = 
\begin{pmatrix}
	\sum
C^{\alpha^1, \ell}_\beta L^\beta \bar{X}_\ell \\ 
\vdots \\
\sum
C^{\alpha^N, \ell}_\beta L^\beta \bar{X}_\ell
\end{pmatrix}.
\]
We multiply this system by the classical adjoint of the matrix 
\[\begin{pmatrix}
	A^{\alpha^1,r_1}_1 & \dots & A^{\alpha^1,r_1}_N \\
	\vdots &\ddots & \vdots\\
	A^{\alpha^N,r_N}_1 & \dots & A^{\alpha^N,r_N}_N 
\end{pmatrix}\]
and obtain for each 
$j=1,\dots, N$ that
\[D(\underline{\alpha},r)  X_j = \sum_{\substack{|\beta|\leq k \\ \ell=1,\dots, N}} D^{\underline{\alpha},r}_{\beta,j } L^\beta \bar X_j, \]
where the $D^{\underline{\alpha},r}_{\beta,j }$ are smooth functions 
on $U$; the right hand side of this equation therefore extends below. 

Hence, for any $\underline{\alpha}$ and $r$ with $|\alpha^j| \leq k$, we have that
$D(\underline{\alpha},r)  X_j$ extends above {\em and} below; in particular, we have that $\WF (X_j) =\emptyset$ so 
that we can conclude that  $D(\underline{\alpha},r)  X_j$ is actually smooth.  Since any 
$\lambda$ in the statement of \autoref{prop:main} can, over $U$, be written as
a smooth linear combination of $D(\underline{\alpha},r) $ with
$|\alpha^j| \leq k$, the proof is finished.
\end{proof}

\section{Proof of the further statements in \autoref{sec:introduction_and_statement_of_results}}
\label{sec:proof_of_theorem_2}
In this section we give the proofs of \autoref{thm:main} and \autoref{cor:main2a}.
The statement of \autoref{thm:main} follows immediately from 
\autoref{thm:multipliers} and \autoref{prop:division}:
By assumption there is a multiplier $\lambda\in\mathcal{S}$ near $p$ 
whose zero set is a real hypersurface near $p$ and $\lambda$ is not flat at $p$. 
\autoref{thm:multipliers} implies that $\lambda X$ is smooth near $p$ and 
since $X$ is assumed to be locally integrable we can apply \autoref{prop:division} to conclude 
that $X$ has also to be smooth near $p$.

In order to prove \autoref{cor:main2a} we now have to show that 
a weakly $k$-nondegenerate real hypersurface $M$ is CR-regular, 
i.e.\ on $M$ there is a multiplier $\lambda$ that can be written in suitable local coordinates as
\begin{equation*}
\lambda(z,\bar{z},s)=s^\ell\psi(z,\bar{z},s)
\end{equation*}
with $\psi$ being a smooth function that does not vanish for $s=0$ and $\ell\in\N$.

By assumption we have that there are coordinates $(z,w)\in\C^n\times\C$
 such that $M$ is given locally by
\begin{equation*}
\imag w=(\real w)^m\varphi (z,\bar{z},\real w)
\end{equation*}
where $m\!\in\!\N$ and $\varphi$ is a smooth real-valued function defined near $0$
with the property that $\varphi_{z^\alpha}(0)\!=\!\varphi_{\bar{z}^\alpha}(0)\!=\!0$ 
for $\lvert\alpha\rvert\leq k$ and
\begin{equation*}
\spanc_\C \{ \varphi_{z{\bar z}^\alpha}  (0,0,0) \colon 0<\lvert\alpha\rvert\leq k\}
= \Cn.
\end{equation*}

A local basis of the CR vector fields on $M$ is given by
\begin{equation*}
L_j=\frac{\partial}{\partial \bar{z}_j}-i\frac{s^m\varphi_{\bar{z}_j}}{1+i(s^m\varphi)_s}\frac{\partial}{\partial s},\qquad 1\leq j\leq n.
\end{equation*}

The characteristic bundle is spanned near the origin by 
\begin{equation*}
\theta =-ds-i\frac{s^m}{1+i(s^m\varphi)_s}\sum_{j=1}^n \varphi_{\bar{z}_j}\,d\bar{z}_j
+i\frac{s^m}{1-i(s^m\varphi)_s}\sum_{j=1}^n \varphi_{z_j}\,dz_j
\end{equation*}
and $\theta$ together with the forms $\omega^j=dz_j$ constitute a local basis of $T^\prime M$.

For simplicity we set
\begin{equation*}
b^j=-i\frac{s^m}{1+i(s^m\varphi)_s}\varphi_{\bar{z}_j}
\end{equation*}
for $1\leq j\leq n$.
If we consider a general holomorphic form 
\begin{equation*}
\eta=\sigma\,\theta + \sum_{j=1}^n\rho_j\omega^j=-\sigma\,ds+\sum_{j=1}^n \sigma b^j
\,d\bar{z}_j+\sum_{j=1}^n \bigl(\sigma \bar{b}_j+\rho^j\bigr)\,dz_j
\end{equation*}
with $\sigma,\rho^j$ being smooth functions on $M$,
then we obtain
\begin{equation}\label{generalderivative}
d\eta(L_\ell,\, .\,)=\bigl((\sigma b^\ell)_s +\sigma_{\bar{z}_\ell}\bigr)\,\theta 
+\sum_{j=1}^n \bigl(L_\ell\rho^j +\sigma (L_\ell \bar{b}^j-\bar{L}_j b^\ell)\bigr)\,\omega^j.
\end{equation}

For a multi-index $\alpha\in\N_0^n$ with length $\lvert\alpha\rvert = r$ we define the following sequence of multi-indices
\begin{align*}
\alpha(1)&=e_1\\
\alpha(2)&=2e_1\\
&\;\;\vdots\\
\alpha(\alpha_1)&=\alpha_1 e_1\\
\alpha(\alpha_1+1)&=\alpha_1 e_1 +e_{2}\\
&\;\;\vdots\\
\alpha(\alpha_1+\alpha_{2})&=\alpha_1e_1+\alpha_2e_2\\
&\;\;\vdots\\
\alpha(\lvert\alpha\rvert)&=\alpha.
\end{align*}
Furthermore let $L_j^t=-\partial_{\bar{z}_j}-b^j\partial_s-b^j_s$ be the formal adjoint of $L_j$ (c.f.\ \cite{Ho}) and $b^{e_j}:=b^j$, $j=1,\dotsc,n$.

Iterative application of \eqref{generalderivative} leads to 
\begin{equation*}
\begin{split}
\mathcal{L}^\alpha\theta &= \bigl((-L^t)^\alpha 1\bigr)\,\theta 
+\sum_{\ell=1}^n \sum_{\nu=1}^{\vert\alpha\rvert} L^{\alpha(\nu)}
\Bigl(\bigl((-L^t)^{\alpha-\alpha(\nu)}1\bigr)
\bigl(L^{\alpha(\nu)-\alpha(\nu-1)}\bar{b}^\ell-\bar{L}_\ell b^{\alpha(\nu)-\alpha(\nu-1)}\bigr)\Bigr)\,\omega^\ell\\
&=A_\theta^\alpha\,\theta+\sum_{\ell=1}^n A^\alpha_\ell\,\omega^\ell.
\end{split}
\end{equation*}
We claim that $A^{\alpha}_\ell=s^m B^{\alpha}_\ell$ 
where $B^{\alpha}_\ell$ are some smooth functions  and 
$B^{\alpha}_\nu(0)=2i\varphi_{\bar{z}^{\alpha}z_\nu}(0)$ for 
$\lvert\alpha\rvert\leq k$.

First, we observe that for $1\leq j,\ell\leq n$  
\begin{equation*}
\begin{split}
L_j\bar{b}^\ell-\bar{L}_\ell b^j&=s^m\Biggl(\frac{i\varphi_{\bar{z}_j z_\ell}
(1+i(s^m\varphi))_s + \varphi_{z_\ell}(s^m\varphi_{\bar{z}_j})_s}{(1+i(s^m\varphi)_s)^2}
+\frac{\varphi_{\bar{z}_j}\bigl((s^m\varphi_{z_\ell})_s(1+i(s^m\varphi)_s)
-is^m\varphi_{z_\ell}(s^m\varphi)_{ss}\bigr)}{(1+i(s^m\varphi)_s)^3}\\
&\quad\;\; +\frac{i\varphi_{\bar{z}_j z_\ell}
(1+i(s^m\varphi))_s + \varphi_{\bar{z}_j}(s^m\varphi_{z_\ell})_s}{(1+i(s^m\varphi)_s)^2}
-\frac{\varphi_{z_\ell}\bigl((s^m\varphi_{\bar{z}_j})_s(1+i(s^m\varphi)_s)
-s^m\varphi_{\bar{z}_j}(s^m\varphi)_{ss}\bigr)}{(1+i(s^m\varphi)_s)^3}\Biggr)\\
&=s^m\lambda^j_\ell
\end{split}
\end{equation*}
and 
$\lambda^j_\ell(0)=2i\varphi_{\bar{z}_jz_\ell}(0)$
by the assumptions on $\varphi$. Furthermore we remark that 
\begin{equation*}
(-L^t)^\beta 1=(-L^t)^{\beta (\lvert\beta\rvert-1)}b^r_s
\end{equation*}
where $r\in\{1,\dotsc,n\}$ is the greatest integer such that $\beta_r\neq 0$.

If we also recall the two simple facts for smooth functions $f,g$: $(s^qf)_s=s^{q-1}f+s^qf_s$ for $q\geq 2$ whereas $(sg)_s=g+sg_s$ we see the following:
If $m\geq 2$ we have
\begin{equation*}
A^{\alpha}_\ell (z,\bar{z},s)=s^m 
\frac{2i\varphi_{\bar{z}^\alpha z_\ell}(z,\bar{z},s)}{1+(s^m\varphi(z,\bar{z},s))_s^2}
+s^{2m-1}R^\alpha_\ell(z,\bar{z},s)=s^m B^{\alpha}_\ell(z,\bar{z},s).
\end{equation*}
On the other hand we obtain for $m=1$ the following representation
\begin{equation*}
A^{\alpha}_\ell(z,\bar{z},s)
=s\frac{2i\varphi_{\bar{z}^\alpha z_\ell}(z,\bar{z},s)}{1+(\varphi(z,\bar{z},s)+s\varphi_s(z,\bar{z},s))^2}
+sS^{\alpha}_\ell(z,\bar{z},s)+s^2T^{\alpha}_\ell(z,\bar{z},s)=sB^{\alpha}_\ell(z,\bar{z},s),
\end{equation*}
where $S^{\alpha}_\ell$ is a sum of products of rational functions 
with respect to $\varphi$ and its derivatives.
Each of these summands contains at least one factor of the form 
$\varphi_{\bar{z}^\beta}$ or $\varphi_{z^\beta}$
with $\lvert\beta\rvert\leq\lvert\alpha\rvert\leq k$ 
and therefore $S^{\alpha}_\ell(0)=0$. 
The claim follows.

By assumption there are multi-indices $\alpha^1,\dots,\alpha^n\neq 0$ of length $\leq k$ such that 
\begin{equation*}
\{\varphi_{z\bar{z}^{\alpha^1}}(0),\dots ,\varphi_{z\bar{z}^{\alpha^n}}(0)\}
\end{equation*}
is a basis for $\C^n$.
Now we choose $\underline{\alpha}=(0,\alpha^1,\dots,\alpha^n)$ and calculate according to 
\eqref{equ:basisfunctions} the multiplier $D(\underline{\alpha})=D(\underline{\alpha}, 1)$
(note that $d=1$):
\begin{align*}
D(\underline{\alpha})&=\det \begin{pmatrix}
1 & 0 & \dots &0\\
A^{\alpha^1}_\theta & A^{\alpha^1}_1 & \dots & A^{\alpha^1}_n\\
\vdots & \vdots & \ddots & \vdots\\
A^{\alpha^n}_\theta & A^{\alpha^n}_1 &\dots & A^{\alpha^n}_n
\end{pmatrix}\\
&=s^{n\cdot m}\det\begin{pmatrix}
1 & 0 & \dots &0\\
A^{\alpha^1}_\theta & B^{\alpha^1}_1 & \dots & B^{\alpha^1}_n\\
\vdots & \vdots & \ddots & \vdots\\
A^{\alpha^n}_\theta & B^{\alpha^n}_1 &\dots & B^{\alpha^n}_n
\end{pmatrix}\\
&=s^{n\cdot m}Q(\underline{\alpha})\\
\intertext{where}
Q(\underline{\alpha})&=\det\begin{pmatrix}
1 & 0 & \dots &0\\
A^{\alpha^1}_\theta & B^{\alpha^1}_1 & \dots & B^{\alpha^1}_n\\
\vdots & \vdots & \ddots & \vdots\\
A^{\alpha^n}_\theta & B^{\alpha^n}_1 &\dots & B^{\alpha^n}_n 
\end{pmatrix}
=\det\begin{pmatrix}
 B^{\alpha^1}_1 & \dots & B^{\alpha^1}_n\\
\vdots & \ddots & \vdots\\
B^{\alpha^n}_1 &\dots & B^{\alpha^n}_n
\end{pmatrix},\\
\intertext{hence}
Q(\underline{\alpha})(0)&=(2i)^n\det \begin{pmatrix}
\varphi_{z\bar{z}^{\alpha^1}}(0)\\
\vdots\\
\varphi_{z\bar{z}^{\alpha^n}}(0)
\end{pmatrix}
\neq 0.
\end{align*}
The proof of \autoref{cor:main2a} is complete.
\section{An Example}\label{sec:example}
In this section we are going to present an example to show that the local integrability condition in 
\autoref{thm:main} and \autoref{cor:main2a}, respectively, is essential 
for the conclusions in these statements to hold.
More precisely, we construct two different infinitesimal diffeomorphisms with distributional coefficents 
 on a real hypersurface in $\C^2$ such that the two diffeomorphisms are not locally integrable.
 We also construct a multiplier such that the products of this multiplier 
 with each diffeomorphism coincide and are smooth.
We further note that the coefficients of both diffeomorphisms are closely related 
to the non-extendable CR distribution for nonminimal CR submanifolds
given by Baouendi and Rothschild \cite{Baouendi-Rothschild}.

We begin with the calculation of the multiplier in a more general setting 
in order to simplify the computation. We will later on restrict ourselves to real hypersurfaces in $\C^2$.
Let $(M,\crb)$ be a $3$-dimensional abstract CR structure of hypersurface type that is generated 
in some coordinates by the vector field
\begin{equation*}
L=\frac{\partial}{\partial \bar{z}}+s^mb(z,\bar{z})\frac{\partial}{\partial s}.
\end{equation*}
The characteristic bundle $T^0M$ is spanned by
\begin{equation*}
\theta= -ds+s^{m}\bar{b}(z,\bar{z})dz+s^{m}b(z,\bar{z})d\bar{z}
\end{equation*}
and thus the forms $\omega =dz$ and $\theta$ form a basis of $T'M$.
We obtain 
\begin{equation*}
d\theta (L,\,.\,)=-2is^m \imag\Bigl(\frac{\partial b}{\partial z}\Bigr)(z,\bar{z})\omega +ms^{m-1}b(z,\bar{z})\theta.
\end{equation*}


We calculate the simplest nontrivial multiplier: for $\alpha^1=0,\ \alpha^2=1$ and $r=(1,1)$
(note that $N=2$ and $d=1$) we have
\begin{equation*}
\begin{split}
D(\underline{\alpha},r)&=\det
\begin{pmatrix}
1 & 0\\
 ms^{m-1}b(z,\bar{z}) &  -2is^m \imag \Bigl(\frac{\partial b}{\partial z}\Bigr)(z,\bar{z}) 
\end{pmatrix}
\\
&=-2is^m \imag\Bigl(\frac{\partial b}{\partial z}\Bigr)(z,\bar{z}).
\end{split}
\end{equation*}

Now let $m=1$, 
$b=-i\frac{\psi_{\bar{z}}}{1+i\psi}$ 
for some smooth real-valued function $\psi$ defined in an open neighbourhood $V$ of $0\in\C$, i.e. $M$ is an embedded real hypersurface in 
$\C^2$ given near the origin by the defining function
\begin{equation*}
\rho(z,\bar{z},w,\bar{w})=\imag w-\real w \cdot\psi(z,\bar{z}).
\end{equation*}
Then the multiplier $D(\underline{\alpha},r)$ from above is of the form
\begin{equation*}
D(\underline{\alpha},r)=
2is\biggl(\frac{\psi_{z\bar{z}}}{\lvert\Psi\rvert^2}-2\frac{\psi_z\psi_{\bar{z}}\psi}{\lvert \Psi\rvert^4}\biggr)=2is\, G(z,\bar{z}),
\end{equation*}
where we have set $\Psi:=1+i\psi$. Note also that $\omega_1=\omega =dz$ and $\omega_2=dw=\Psi ds+is\psi_zdz+is\psi_{\bar{z}}d\bar{z}$ is an alternative basis for $T^\prime M$ in this case.

Since $M$ is an real hypersurface in $\C^2$ we have the following decomposition of an open neighbourhood $\Omega$ of $0\in\C^2$
\begin{equation*}
\Omega =U_+\cup M\cup U_-
\end{equation*}
with $U_+=\{(z,w)\in\Omega \colon \rho(z,\bar{z},\bar{z},w,\bar{w}) >0\}$ and 
$U_-=\{(z,w)\in\Omega \colon \rho(z,\bar{z},w,\bar{w})<0\}$ being open subsets of $\Omega$.
We shall also assume that $\Omega\cap(\C\times\{0\})=V\times\{0\}$.

If we consider the holomorphic function
\begin{equation*}
F:\; (z,w)\longmapsto \frac{1}{w}
\end{equation*}
on $\C\times\C\!\setminus\!\{0\}$ then we see that $F$ is of slow growth for $w\rightarrow 0$ 
on both $U_+$ and $U_-$.
We write $u_+=b_+F$ for the boundary value of $F_{|U_+}$ and $u_-=b_-F$ 
for the boundary value of $F_{|U_-}$, respectively. Note that by the Plemelj-Sokhotski jump relations 
(see, e.g., \cite{Duistermaat}) we have
\begin{equation*}
u_0=u_+-u_-=-\frac{2\pi i}{\Psi}(1\otimes \delta).
\end{equation*}
Note also that $u_0$ is essentially (up to the factor $-2\pi i$) the non-extendable CR distribution from 
\cite{Baouendi-Rothschild}, c.f.\ also \cite{Baouendi:1999uy}, for the hypersurface $M$. 

We claim that $\WF u_+ =\R_{+}\theta\vert_{V\times\{0\}}$ and 
$\WF u_-=\R_-\theta\vert_{V\times\{0\}}$, respectively: 
Note that $u_+$ and $u_-$ are smooth outside $V\!\times\!\{0\}\subset M$ and that 
$\WF u_0=(\R\!\setminus\!\{0\})\,\theta\vert_{V\times\{0\}}$.
Furthermore we know that $\WF u_+$ and $\WF u_-$ must each be contained in 
$(\R\!\setminus\!\{0\})\theta$ since both are CR distributions.
However, since $u_+$ extends holomorphically to $U_+$ it follows that 
$\WF u_+\cap \R_-\theta =\emptyset$ (see e.g.\ \cite{Lamel:2004hh}) and 
by symmetry we have also $\WF u_-\cap \R_+\theta =\emptyset$.
Now let $p=(z,0)\in V\!\times\!\{0\}$ and suppose that, e.g., $\R_+\theta_p\cap\WF u_+=\emptyset$.
Then we would have that $\R_+\theta_p\cap\WF u_0=\emptyset$ which is obviously a contradiction to above.

We consider the following vector fields with distributional coefficients
\begin{align*}
X_+&= u_+\frac{\partial}{\partial z}{\Big|}_M +\bar{u}_+\frac{\partial}{\partial \bar{z}}{\Big|}_M\\
\intertext{and}
X_-&=u_-\frac{\partial}{\partial z}{\Big|}_M+\bar{u}_-\frac{\partial}{\partial \bar{z}}{\Big|}_M.
\end{align*}
We claim that both vector fields constitute infinitesimal CR diffeomorphisms on $M$ if 
\begin{equation*}
\frac{\partial \psi}{\partial x}=\psi\frac{\partial \psi}{\partial y}
\end{equation*} 
where  $z=x+iy$.
We show this for $X_+$, the argument for $X_-$ is completely analagous of course.
First we see that $X_+$ is real since
\begin{equation*}
X_+=\real u_+\frac{\partial}{\partial x}{\Big|}_M+\imag u_+\frac{\partial}{\partial y}{\Big|}_M.
\end{equation*}
 Furthermore note that the regular distributions ($\nu >0$)
\begin{align*}
u_\nu=\frac{1}{s\Psi +i\nu}
\end{align*}
on $M$ converge to $u_+$ in $\D^\prime$ for $\nu\rightarrow 0$.
We have 
\begin{equation*}
\begin{split}
X_+\rho &=-s\psi_x\real u_+ -s\psi_y\imag u_+\\
&=\lim_{\nu\rightarrow 0}\bigl(-s\psi_x\real u_\nu-s\psi_y\imag u_\nu\bigr)\\
&=\lim_{\nu\rightarrow 0}\biggl(\frac{-s^2(\psi_x-\psi\psi_y) +s\nu}{s^2+(s\psi+\nu)^2}\biggr)\\
&=\lim_{\nu\rightarrow 0}s\nu \lvert u_\nu\rvert^2=0
\end{split}
\end{equation*}
with convergence in $\D^\prime$.
Hence $X_+\in\D^\prime(M, TM)$. We conclude further
\begin{align*}
L\bigl(\omega_1 (X_+)\bigr)&=Lu_+=0, \\
L\bigl(\omega_2(X_+)\bigr)&=0\\
\intertext{and since $d\omega_j=0,$ $(j=1,2)$}
d\omega_1 (L,X_+)&=0,\\
d\omega_2 (L,X_+)&=0.
\end{align*}

Since $\omega_1(X_+)=\omega_1(X_-)=u_+$, $\omega_2(X_+)=\omega_2(X_+)=0$ and 
 $\omega_1(X_-)=u_-$ all the assumptions of \autoref{thm:multipliers} are satisfied for 
 both $X_+$ and $X_-$.
 
Indeed
\begin{equation*}
D(\underline{\alpha},r)u_+=D(\underline{\alpha},r)u_-=2i\frac{G(z,\bar{z})}{\Psi(z,\bar{z})}\in\E(M)
\end{equation*}
hence $D(\underline{\alpha},r)X_+=D(\underline{\alpha},r)X_-\in\E$.
Note also that $D(\underline{\alpha},r)u_0=0$.

\bibliographystyle{plain}
\bibliography{fl15}
\end{document}